\documentclass{amsart}
\usepackage[T1]{fontenc}
\usepackage{amsthm}
\usepackage{amssymb}
\usepackage{a4wide}
\usepackage{graphicx}
\usepackage{graphics}
\usepackage{epsfig}
\usepackage{multirow}

\numberwithin{equation}{section} 
\numberwithin{figure}{section} 
\theoremstyle{plain}
\theoremstyle{plain}
\newtheorem{thm}{Theorem}[section]
\theoremstyle{plain}
\newtheorem{prop}[thm]{Proposition}
\theoremstyle{remark}

\theoremstyle{plain}

\theoremstyle{plain}

\theoremstyle{plain}

\theoremstyle{plain}

\begin{document}

\title{Numerical Computation of approximate Generalized Polarization Tensors}

\author{Yves Capdeboscq}

\address{Mathematical Institute, 24-29 St Giles, OXFORD OX1 3LB, UK}

\email{capdeboscq@maths.ox.ac.uk}

\author{Anton Bongio Karrman}

\address{Applied and Computational Mathematics, The California Institute of
Technology, Pasadena CA 91125, USA}

\author{Jean-Claude N\'ed\'elec}

\address{CMAP, CNRS  UMR 7641, Ecole Polytechnique, 91128 Palaiseau, France}

\begin{abstract}
In this paper we describe a method to compute Generalized Polarization
Tensors. These tensors are the coefficients appearing in the multipolar
expansion of the steady state voltage perturbation caused by an inhomogeneity
of constant conductivity. As an alternative to the integral equation
approach, we propose an approximate semi-algebraic method which is
easy to implement. This method has been integrated in a Myriapole, a  
matlab routine with a graphical interface which makes such computations 
available to non-numerical analysts.
\end{abstract}
\maketitle

\section{Introduction}

The classical electrical impedance tomography problem in two dimensions
can be described as follows. Suppose that $\Omega\subset\mathbb{R}^{2}$
is a bounded simply connected domain. Given a conductivity map $\gamma \in L^{\infty}(\Omega)$,
suppose that there exists a constant $c>0$ such that $c^{-1}\leq \gamma <c$.
Define the Dirichlet-to-Neumann map $\Lambda:H^{1/2}(\partial\Omega)\to H^{-1/2}(\partial\Omega)$
by \[
\Lambda(\phi):=\left.\gamma \frac{\partial u}{\partial\nu}\right|_{\partial\Omega},\]
where $u\in H^{1}(\Omega)$ is the voltage potential associated with
$\gamma$, that is, the unique solution of \begin{eqnarray*}
\mbox{div}(\gamma\nabla u) & = & 0\mbox{ in }\Omega\\
u & = & \phi\mbox{ on }\partial\Omega.\end{eqnarray*}
The electrical tomography problem is then to reconstruct $\gamma$ from
$\Lambda$. The fact that $\gamma$ is uniquely determined by $\Lambda$
has recently been established in dimension two \cite{ASTALA-PAIVARINTA-06},
and this question is still open in dimension three. It is known however
that, without additional assumptions on the conductivity, this problem
is extremely ill-posed \cite{ALESSANDRINI-97}. This ill-posedness
can be considerably reduced in several cases of practical importance
by the introduction of a priori information on the form of the conductivity
to be recovered. One such problem is the following: given complete
(or incomplete) knowledge of the Dirichlet-to-Neumann map, estimate
the internal conductivity profile under the assumption that $\gamma$ is constant
except at a finite number of small inhomogeneities. This problem,
and some variants, has been investigated by many authors (to name
a few, see \cite{AMMARI-KANG-04,BRULH-HANKE-VOGELIUS-03,CAPDEBOSCQ-VOGELIUS-03A,CEDIO-FENGYA-MOSKOW-VOGELIUS-98,FRIEDMAN-VOGELIUS-89}
and references therein). One important step in this estimation procedure
is the derivation of an asymptotic expansion of the voltage potential
$u$ in terms of the volume of the inhomogeneities. For an isolated
inhomogeneity, the inhomogeneity dependent parameters in the first
order of this expansion are called the Polarization Tensor (a constant
$2\times2$ matrix). The generalization of this tensor includes higher 
order terms dependent on the inhomogeneity via what is referred to as 
the Generalized Polarization Tensor (GPT) (see \cite{AMMARI-KANG-MMS-03,AMMARI-KANG-SIMA-03,AMMARI-KANG-07}) 
and was generalized to the case of linear elasticity (see \cite{AMMARI-KANG-NAKAMURA-TANUMA-02}). 

These tensors appear in various contexts. In homogenization theory,
they provide asymptotic expansions for dilute composites (see \cite{MILTON-02,AMMARI-KANG-TOUIBI-05}).
Very recently, GPT's have been used to reconstruct shape information
from boundary measurements \cite{AMMARI-KANG-LIM-ZRIBI-11} and to
construct approximate cloaking devices \cite{AMMARI-KANG-LEE-LIM-11}.
In these last two papers, the authors consider the following conductivity
transmission problem: \[
\begin{cases}
\mbox{div}\left(\left(\mathbf{1}_{D}\gamma_{1}+\left(1-\mathbf{1}_{D}\right)\gamma_{0}\right)\nabla u\right)=0 & \mbox{ in }\mathbb{R}^{2},\\
u(x)-H(x)=O(\left|x\right|^{-1}) & \mbox{ as }\left|x\right|\to\infty,\end{cases}\]
where $\gamma_{1}>0$ and $\gamma_{0}>0$ are constants, $\mathbf{1}_{D}$
is the characteristic function of a bounded domain $D\subset\mathbb{R}^{2}$
with a Lipschitz boundary and a finite number of connected components,
and $H$ is a given harmonic function. In the absence of the inclusion
$D$, $u=H$. Then, the far-field perturbation of the voltage potential
created by $D$ is given by \begin{equation}
u(x)-H(x)=-\sum_{|\alpha|,|\beta|=1}^{\infty}\frac{(-1)^{|\alpha|}}{\left|\alpha\right|!\left|\beta\right|!}\partial^{\alpha}\Gamma(x)M_{\alpha\beta}\partial^{\beta}H\left(x_{0}\right)\mbox{ as }\left|x\right|\to\infty,\label{eq:expan-gpt-1}\end{equation}
where $x_{0}$ is the center of mass of $D$, and $\Gamma$ is the
fundamental solution of the Laplacian, \[
\Gamma(x):=- \frac{1}{2\pi}\ln\left|x\right|.\]
The parameters $\alpha=(\alpha_{1},\alpha_{2})$ and $\beta=(\beta_{1},\beta_{2})$
are multi-indices, and $|\alpha|=\alpha_{1}+\alpha_{2}$. The real
valued tensors $M_{\alpha\beta}$, which depend on $D$ and on the conductivity
contrast $\gamma_{1}/\gamma_{0}$ are the 
so-called GPT. In bounded
domains, expansions similar to \eqref{eq:expan-gpt-1} are also available,
where $\Gamma$ is replaced by a suitable Green's Function (see \cite{AMMARI-KANG-04}).
To fix ideas, let us assume $x_{0}$ is the origin. As it is observed
in \cite{AMMARI-KANG-LEE-LIM-11}, the expansion \eqref{eq:expan-gpt-1}
can be simplified when expressed in terms of harmonic polynomials
which are real or imaginary parts of $(x_{1}+ix_{2})^{n}$ where $x=(x_{1},x_{2})$.
Introducing \[
a_{n}(x):=r^{n}\cos(n\theta)=\Re\left(\left(x_{1}+ix_{2}\right)^{n}\right),\quad b_{n}(x):=r^{n}\sin(n\theta)=\Im\left(\left(x_{1}+ix_{2}\right)^{n}\right),\]
if $H$ can be written in the form\[
H(x)=H(0)+\sum_{n=1}^{\infty}\alpha_{n}(H)a_{n}(x)+\beta_{n}(H)b_{n}(x),\]
then it is shown in \cite{AMMARI-KANG-LEE-LIM-11} that the expansion
\eqref{eq:expan-gpt-1} can be written under the form 
\begin{eqnarray*}
u(x)-H(x) & = & 
 \sum_{m=1}^{\infty}\frac{a_m(x)}{2\pi m|x|^{2m}}\sum_{n=1}^{\infty}\left(M\left(a_{m,}a_{n}\right)\alpha_{n}(H)+M\left(a_{m,}b_{n}\right)\beta_{n}(H)\right)\\
 &+  & \sum_{m=1}^{\infty}\frac{b_m(x)}{2\pi m |x|^{2m}}\sum_{n=1}^{\infty}\left(M\left(b_{m,}a_{n}\right)\alpha_{n}(H)+M\left(b_{m,}b_{n}\right)\beta_{n}(H)\right),\mbox{ as }|x|\to\infty.
\end{eqnarray*}
The authors of \cite{AMMARI-KANG-LEE-LIM-11}  refer to $M(P,Q)$, where $P$ and $Q$ are any of the harmonic
polynomials $a_{n}$ or $b_{n}$, as the contracted GPT. Such an expansion
is particularly relevant if the inclusion is observed on a sphere
$S_{R}$ (of radius $R$) centered at the origin. In such a case, it
is well known that $a_{n}$ and $b_{n}$ form a basis of $L^{2}(S_{R})$
- this is the Fourier series decomposition. The tensor coefficients $M\left(P,Q\right)$
are given by the formula
\begin{equation}\label{eq:def-MPQ}
M\left(P,Q\right)=\int_{\partial D}P(x){\phi^{Q}}(x)\, d\sigma(x), 
\end{equation}
where ${\phi^{Q}}$ solves the transmission problem
\begin{eqnarray}
\Delta {\phi^{Q}} & = & 0\mbox{ in }D\mbox{ and }\mathbb{R}^{2}\setminus\overline{D},\nonumber \\
\left.{\phi^{Q}}\right|_{+} & = & \left.{\phi^{Q}}\right|_{-}\mbox{ on }\partial D,\nonumber \\
\left.\frac{\partial {\phi^{Q}}}{\partial\nu}\right|_{+}-\frac{\gamma_{1}}{\gamma_{0}}\left.\frac{\partial {\phi^{Q}}}{\partial\nu}\right|_{-} & = & \frac{\partial Q}{\partial\nu}\mbox{ on }\partial D,\label{eq:trans1}\\
{\phi^{Q}} & \to & 0\mbox{ as }|x|\to\infty,\nonumber 
\end{eqnarray}
where $\nu$ is the normal pointing outside of $D$, and where the subscript $+$ (resp. $-$) refers to the limit taken along the normal from the outside 
(resp. the inside) of $D$. 

The subject of this paper is to detail a simple algorithm to compute
approximate values of the contracted GPT $M(P,Q)$. We focus on the
case when $\gamma_{1}$ is constant, and $D$ is a simply connected
inclusion with a $C^{1}$ boundary. In Section~\ref{sec:var}, we show that this
problem can be written in a weak form as a system of boundary integral equations on $\partial D^2:=\partial D \times \partial D$.
Such an approach is well-known amongt numerical analysts interested in integral equations. 
We use an integration by parts formula that can be found in \cite{COLTON-KRESS-98,NEDELEC-01}. 
In Section~\ref{sec:linear}, in the spirit of mixed finite-elements methods, we decouple the 
boundary voltage potential from the normal flux and approximate them in the
following way: 
\begin{eqnarray*}
\left.{\phi^{Q}}\right|_{-} 
& \in & \mbox{span}\left(a_{1},b_{1},\ldots,a_{N_{1}},b_{N_{1}}\right),\\
\left.\frac{\partial {\phi^{Q}}}{\partial\nu}\right|_{-} 
& \in & \mbox{span}\left(\frac{\partial}{\partial\nu}a_{1},\frac{\partial}{\partial\nu}b_{1},\ldots,\frac{\partial}{\partial\nu}a_{N_{2}},\frac{\partial}{\partial\nu}b_{N_{2}}\right),
\end{eqnarray*}
for a given a fixed $N_{1}$ and $N_{2}$, greater than or equal to the
order of $Q$. We verify that the corresponding linear system is invertible. 
There are two advantages of this approach. The first is that since $a_{n}$ and $b_{n}$ are 
explicit polynomials, their gradient can also be computed explicitly. The precision with 
which the normal is calculated depends on the parameterization of $D$, which is independent
of $N_{1}$ and $N_{2}$. The second is that the problem then involves
the computation of $4 (N_{1}+N_{2})^{2}$ integrals on $\partial D^2$,
and then the inversion of a $4 (N_{1}+N_{2})^{2}$ matrix. In other
words, this decouples the order of the contracted GPT from the discretization
of the boundary $D$. In Section~\ref{sec:num}, we discuss the results obtained
for disks and ellipses, for which the exact values of the contracted GPT
are known. Finally, in Section~\ref{sec:gui} , we present a Matlab routine (named Myriapole
on the Matlab Central exchange repository) that implements this approach.

\section{\label{sec:var}Weak boundary integral formulation of the transmission problem}
In this section, we derive the following variational formulation satisfied by the transmission problem.
\begin{prop}\label{pro:var}
Let ${\phi^{Q}}$ be the solution of the transmission problem \eqref{eq:trans1}. Introducing the notations
$$
{u^{Q}}:=\left. {\phi^{Q}} \right|_{\partial D} \in H^{1/2}(\partial D),\quad {v^{Q}}:=\left. \frac{\partial {\phi^{Q}}}{\partial \nu}\right|_{ -} \in H^{-1/2}(\partial D), \quad k=\frac{\gamma_1}{\gamma_0},
$$
the pair $({u^{Q}},{v^{Q}})$ satisfies for any $\phi\in H^{1/2}(\partial D)$, 
\begin{eqnarray}\label{eq:eq1-var}
\left(k+1\right) \int_{\partial D^2}  \frac{\partial {u^{Q}}}{\partial \tau} (y) \frac{\partial \phi}{\partial \tau} (x) \Gamma (x-y) \, d\sigma(x) d\sigma(y) 
+ 2 k \int_{\partial D^2}  {v^{Q}}(y) \phi(x) \frac{\partial \Gamma (x-y)}{\partial \nu_x} \, d\sigma(x) d\sigma(y) 
  \nonumber\\
= \frac{1}{2} \int_{\partial D} \frac{\partial Q}{\partial \nu}(x) \phi(x) \, d\sigma (x) 
- \int_{\partial D^2}  \frac{\partial Q}{\partial \nu}(y) \phi(x) \frac{\partial \Gamma (x-y)}{\partial \nu_x} \, d\sigma(x) d\sigma(y) \nonumber,
\end{eqnarray}
and for all $\psi \in H^{1/2}(\partial D)$, 
\begin{eqnarray}
\left(k+1 \right) \int_{\partial D^2}  {v^{Q}}(y) \psi(x) \Gamma (x-y) \, d\sigma(x) d\sigma(y)
 - 2 \int_{\partial D^2} {u^{Q}}(y) \psi(x) \frac{\partial \Gamma (x-y)}{\partial \nu_y} \, d\sigma(x) d\sigma(y) 
  && \nonumber\\
 = -\int_{\partial D^2} \frac{\partial Q}{\partial \nu}(y) \psi(x) \Gamma (x-y) \, d\sigma(x) d\sigma(y). \label{eq:eq2-var}
\end{eqnarray}
\end{prop}
\begin{proof}
In order to solve the  transmission problem of \eqref{eq:trans1}, we separate the interior potential from the exterior by defining the following:
\begin{equation}\label{eq:defui}
{u^{Q}}(x) :=  {\phi^{Q}}(x) \, \text{ for }\, x \in D, \mbox{ and } {\phi^{Q}}_{\rm ext}(x)  :=  {\phi^{Q}}(x)  \text{ for } \, x \in \mathbb{R}^2 \setminus \overline{D}.
\end{equation}
We then introduce the gradient fields as
\begin{equation}\label{eq:defvi}
{\mathbf{v}^{Q}}_{\rm int}(x)  :=  \nabla {u^{Q}}(x), \, \mbox{  and } \, {\mathbf{v}^{Q}}_{\rm ext}(x)  :=  \nabla {\phi^{Q}}_{\rm ext}(x),
\end{equation}
and set $v^{Q}= {\mathbf{v}^{Q}}_{\rm int}\cdot \nu$ on $\partial D$.
A representation of the solution using Green's identities leads to the expression of the interior solution and its 
gradient in terms of first and double layer potentials. For all $x\in D$, we have
\begin{eqnarray*}
{u^{Q}}(x) & = & \int_{\partial D} {u^{Q}}(y) \frac{\partial \Gamma (x-y)}{\partial \nu_y}  - v^{Q}(y) \Gamma (x-y)   \, d\sigma (y),\\ 
{\mathbf{v}^{Q}}_{\rm int}(x) & = & \int_{\partial D} {u^{Q}}(y) \nabla_x  \left(\frac{ \partial \Gamma (x-y)}{\partial \nu_y}\right)  
-  v^{Q}(y)  \nabla_x \Gamma (x-y)   \, d\sigma (y),
\end{eqnarray*}
As $x \rightarrow \partial D$ from the interior, using the jump conditions satisfied by single and double layer potentials \cite{FOLLAND-95}, we obtain for $x \in \partial D$,
\begin{eqnarray}
\label{eq:a}
\frac{1}{2}{u^{Q}}(x) &=&\int_{\partial D} {u^{Q}}(y) \frac{\partial \Gamma (x-y)}{\partial \nu_y}  - v^{Q}(y) \Gamma (x-y)   \, d\sigma (y), \\
\label{eq:b}
\frac{1}{2} v^{Q} (x) &=& \int_{\partial D} {u^{Q}}(y) \frac{ \partial^2 \Gamma (x-y)}{\partial \nu_x \partial \nu_y} 
- v^{Q}(y) \frac{\partial \Gamma (x-y)}{\partial \nu_x} \, d\sigma (y).
\end{eqnarray}
Similarly, for all $x\in  \mathbb{R}^2 \setminus \overline{D} $ we have
\begin{eqnarray*}
{\phi^{Q}}_{\rm ext}(x) & = &  -\int_{\partial D} {\phi^{Q}}_{\rm ext}(y) \frac{\partial \Gamma (x-y)}{\partial \nu_y}  
- \left({\mathbf{v}^{Q}}_{\rm ext}(y) \, \cdot \, \nu \right)\Gamma (x-y)   \, d\sigma (y), \\ 
{\mathbf{v}^{Q}}_{\rm ext}(x) & = &  -\int_{\partial D} {\phi^{Q}}_{\rm ext}(y) \nabla_x  \left(\frac{ \partial \Gamma (x-y)}{\partial \nu_y}\right)  
 - \left({\mathbf{v}^{Q}}_{\rm ext}(y) \, \cdot \, \nu \right) \nabla_x \Gamma (x-y)   \, d\sigma (y).
\end{eqnarray*}
And passing to the limit as $x \rightarrow \partial B$ for the exterior, we obtain 
\begin{eqnarray}
\label{eq:c}
\frac{1}{2}{\phi^{Q}}_{\rm ext}(x)  &=&  - \int_{\partial D} {\phi^{Q}}_{\rm ext}(y) \frac{\partial \Gamma (x-y)}{\partial \nu_y}  - \left({\mathbf{v}^{Q}}_{\rm ext}(y) \, \cdot \, \nu \right) \Gamma (x-y)   \, d\sigma (y), \\
\label{eq:d}
 \frac{1}{2} {\mathbf{v}^{Q}}_{\rm ext}(x) \, \cdot \, \nu  &=&    - \int_{\partial D} {\phi^{Q}}_{\rm ext}(y) \frac{ \partial^2 \Gamma (x-y)}{\partial \nu_x \partial \nu_y}
 -  \left({\mathbf{v}^{Q}}_{\rm ext}(y) \, \cdot \, \nu \right) \frac{\partial \Gamma (x-y)}{\partial \nu_x} \, d\sigma (y), 
\end{eqnarray}
We subtract \eqref{eq:c} from \eqref{eq:a}  and use \eqref{eq:trans1} and \eqref{eq:defui}  to obtain for $x \in \partial D$,
\begin{eqnarray}\label{eq:test1}
2 \int_{\partial D} {u^{Q}}(y) \frac{\partial \Gamma (x-y)}{\partial \nu_y} \, d\sigma (y) - \left(1+ k \right) \int_{\partial D} v^{Q}(y) \Gamma (x-y) \, d\sigma (y) \\
 = \int_{\partial D} \frac{\partial Q}{\partial \nu}(y) \Gamma (x-y) \, d\sigma (y). \nonumber
\end{eqnarray}
Similarly, from  \eqref{eq:b} and \eqref{eq:d} we have,  for $x \in \partial D$,
\begin{eqnarray}\label{eq:test2}
2 k \int_{\partial D} v^{Q}(y) \frac{\partial \Gamma (x-y)}{\partial \nu_x} \, d\sigma (y) - \left(1+k \right) \int_{\partial D} {u^{Q}}(y) \frac{ \partial^2 \Gamma (x-y)}{\partial \nu_x \partial \nu_y} \, d\sigma (y) \\
= \frac{1}{2} \frac{\partial Q}{\partial \nu}(x) - \int_{\partial D} \frac{\partial Q}{\partial \nu}(y)  \frac{\partial \Gamma (x-y)}{\partial \nu_x}  \, d\sigma (y). \nonumber
\end{eqnarray}
We test \eqref{eq:test1} against $\phi \in H^{1/2}(\partial D) $ and obtain
\begin{eqnarray*}
2 k \int_{\partial D^2}  {v^{Q}}(y) \phi(x) \frac{\partial \Gamma (x-y)}{\partial \nu_x} \, d\sigma(x) d\sigma(y) +\left(1+k\right) \int_{\partial D^2}  \frac{\partial {u^{Q}}}{\partial \tau} (y) \frac{\partial \phi}{\partial \tau} (x) \Gamma (x-y) \, d\sigma(x) d\sigma(y) \\ 
= \frac{1}{2} \int_{\partial D} \frac{\partial Q}{\partial \nu}(x) \phi(x) \, d\sigma (x) 
- \int_{\partial D^2}  \frac{\partial Q}{\partial \nu} (y)\phi(x) \frac{\partial \Gamma (x-y)}{\partial \nu_x} \, d\sigma(x) d\sigma(y).
\end{eqnarray*}
We have used the following integration by parts formula proved in \cite{COLTON-KRESS-98,NEDELEC-01},
$$
-\int_{\partial D^2} {u^{Q}}(y) \phi(x) \frac{ \partial^2 \Gamma (x-y)}{\partial \nu_x \partial \nu_y} \, d\sigma(x) d\sigma(y) = \int_{\partial D^2}  \frac{\partial {u^{Q}}}{\partial \tau} (y) \frac{\partial \phi}{\partial \tau} (x) \Gamma (x-y) \, d\sigma(x) d\sigma(y).
$$
Finally, we test  \eqref{eq:test2}  against $\psi \in H^{1/2}(\partial D)$ to obtain
\begin{eqnarray*}
2 \int_{\partial D^2} {u^{Q}}(y) \psi(x) \frac{\partial \Gamma (x-y)}{\partial \nu_y} \, d\sigma(x) d\sigma(y)
-\left(1+k \right) \int_{\partial D^2}  {v^{Q}}(y) \psi(x) \Gamma (x-y) \, d\sigma(x) d\sigma(y) \\
 = \int_{\partial D^2} \frac{\partial Q}{\partial \nu}(y) \psi(x) \Gamma (x-y) \, d\sigma(x) d\sigma(y).
\end{eqnarray*}
which concludes the proof.
\end{proof}
To conclude this section, we provide two equivalent formulations for the tensor $M(P,Q)$ in terms of the unknowns $u^Q$ and $v^Q$.
\begin{prop}\label{pro:formulaMPQ}
The contracted GPT defined in \eqref{eq:def-MPQ} is also given by
$$
M(P,Q)  =  (k-1)\int_{\partial D} P(x) \left(\frac{\partial Q}{\partial \nu}(x)+(k-1)v^Q(x) \right) \, d\sigma(x), 
$$
and
$$
M(P,Q)=  (k-1)\int_{\partial D}  \frac{\partial P}{\partial \nu}(x) \left(Q+(k-1) u^Q(x)\right) \, d\sigma(x).
$$
\end{prop}
\begin{proof}
The first formula is proved in \cite[Lemma 3.3]{AMMARI-KANG-04}. The latter is obtained after an integration by parts of the former.
\end{proof}

\section{\label{sec:linear}Approximate resolution method}

To solve the system given in \eqref{eq:eq1-var}, we restrict it to a finite-dimensional space. More precisely,
given $N_1$ and $N_2$ two positive integers, we set $n_p:= 2N_1$ and $n_q:= 2N_2$ and define
$$
P_i:=
    \begin{cases}
     a_{\frac{i+1}{2}} &\mbox{ when } i\leq n_{p}  \mbox{ is odd}, \\
     b_{\frac{i}{2}} &\mbox{ when } i\leq n_{p}  \mbox{ is even},
    \end{cases}
$$
and
$$
Q_i:=\begin{cases}
     \nabla a_{\frac{i+1}{2}} \cdot \nu &\mbox{ when } i\leq n_{q}  \mbox{ is odd}, \\
     \nabla b_{\frac{ i}{2}} \cdot &\mbox{ when }  i\leq n_{q}  \mbox{ is even}, 
    \end{cases}
$$
and we consider \eqref{eq:eq1-var} as a system posed on 
$$
\mathcal{H}=\mbox{span}\left(P_1,\ldots,P_{n_{p}}\right) \times  \left(Q_1,\ldots,Q_{n_{q}}\right).
$$
We write 
$$
u^Q=\sum_{n=1}^{2 N_1} c_n P_n,\mbox{ and } v^Q=\sum_{n=1}^{2N_2} d_n Q_n,
$$
and thus \eqref{eq:eq1-var} becomes a linear system of the form $\mathcal{M} X = \mathcal{B}$, where $X$ is a $(n_p +n_q)\times n_q $ matrix, 
where for $q\leq k\leq N_2$, the $k$-th column contains the coefficients the of $(u^{a_k},v^{a_k})$ and the $N_2+k$-th column the coefficients of
 $(u^{b_k},v^{a_k})$.  
For $1\leq i,j \leq n_{p}$,
$$
M_{i,j} =  (k+1) \mathcal{P}_{i,j}, \mbox{ with } 
\mathcal{P}_{i,j} = \int_{\partial D^2}  \frac{\partial {P_i}}{\partial \tau} (y) 
\frac{\partial {P_j}}{\partial \tau} (x) \Gamma (x-y) \, d\sigma(x) d\sigma(y).
$$
For $1\leq i \leq  n_{p}$ and $n_{p} +1 \leq j \leq n_{p} +n_{q}$
$$
M_{i,j} = 2 k \mathcal{N}_{i,j}, \mbox{ with } 
\mathcal{N}_{i,j} =  \int_{\partial D^2}  Q_{j-n_{p}}(y) P_{i}(x) \frac{\partial \Gamma (x-y)}{\partial \nu_x} \, d\sigma(x) d\sigma(y).
$$
For $n_{p}+1 \leq i \leq n_{p}+n_{q}$, and $1\leq j \leq n_{p}$,
$$
M_{i,j} = - 2 \int_{\partial D^2}  Q_{j-n_{p}}(x) P_{j}(y) \frac{\partial \Gamma (x-y)}{\partial \nu_x} \, d\sigma(x) d\sigma(y)  = -2 \mathcal{N}_{j,i}.
$$
For $n_{p}+1 \leq i,j \leq n_{p}+n_{q}$, 
$$
M_{i,j} = (k+1) \mathcal{Q}_{i,j}, \mbox{ with } \mathcal{Q}_{i.j} = \int_{\partial D^2}  Q_{j-n_{p}}(x)Q_{i-n_{p}}(y) \Gamma (x-y) \, d\sigma(x) d\sigma(y).
$$
The right-hand-side term $\mathcal{B}$ is a  $(n_p +n_q)\times n_q $ matrix, given  when  $1\leq i \leq n_p$ by
$$
\mathcal{B}_{i,j}= 
\frac{1}{2} \int_{\partial D} Q_{j}(x) P_{i}(x) \, d\sigma (x) 
- \int_{\partial D^2}  Q_{j}(y) P_i(x) \frac{\partial \Gamma (x-y)}{\partial \nu_x} \, d\sigma(x) d\sigma(y),
$$
and when $n_{p}+1\leq i \leq n_p +n_q$,
$$
\mathcal{B}_{i,j}= -\int_{\partial D^2} Q_{j}(y) Q_{j-n_{p}}(x) \Gamma (x-y) \, d\sigma(x) d\sigma(y).
$$

Note that both $\mathcal{P}$ and $\mathcal{Q}$ are symmetric, negative definite provided $|x-y|<\kappa<1$, that is, provided the perimeter of $D$ is smaller than $2 \kappa$. 
This assumption is not restrictive: we can scale $D$ to be sufficiently small for the computations. In that case, since we established that the 
matrix $\mathcal{M}$ is of the form
$$
\mathcal{M} = \left [
\begin{array}[c]{cc}
\left(k+1\right) \mathcal{P} & 2 k \mathcal{N} \\ \\
- 2 \mathcal{N}^T & \left(k+1\right) \mathcal{Q}
\end{array}
\right ],
$$ 
it has as a positive determinant and therefore this linear-system is well-posed.
For any $i$ and $j$ such that $\max(i,j)\leq \min(n_p,n_q)$, the contracted GPT coefficients are then computed using Proposition~\ref{pro:formulaMPQ}, namely
$$
M(P_i,P_j) = (k-1) \int_{\partial D} P_i (x) \left( Q_{j}(x) + (k-1) \sum_{m=1}^{n_q} X_{m,n_{p}+j} Q_{m}(x)\right) d\sigma(x),
$$           
or equivalently by
$$
M(P_i,P_j) = (k-1) \int_{\partial D} Q_i (x) \left( P_{j}(x) + (k-1) \sum_{m=1}^{n_q} X_{m,j} P_{m}(x)\right) d\sigma(x).
$$
For a fixed pair $(i,j)$, increasing the number $n_p$ and $n_q$ improves the accuracy of the method, to a point.
Note that $P$, $Q$ and $\Gamma$ are explicitly given in terms of $x$ and $\nu(x)$. To compute the integrals which appear in $\mathcal{M}$, $\mathcal{B}$ and finally for 
$M$, we mostly need to choose a parameterization of $\partial D$, and to compute its normal. 
If $\partial D$ is given by an analytic formula, these integrals can all be computed using symbolic computation software. In this sense, the method is
partially algebraic. For non-explicit curves $\partial D$, the only non-routine step is the evaluation of the double integrals involving $\Gamma(x-y)$ and 
$\nabla \Gamma (x-y) \cdot \nu(x)$ on the line-segments (or on the arcs) where $x-y$ cancels (or is smaller than a threshold). However, approximations for these integrals
 can be computed "by hand," thanks to the explicit form of the integrands.
\section{\label{sec:num}Numerical results for disks and ellipses}

In order to test our method, we consider geometries for which exact formulae for the contracted GPT are known, 
that is, disks and ellipses.
If the inclusion is a disk, the contracted GPT is diagonal, and given by
\begin{equation}\label{eq:fordisk}
M(P_{2i-1},P_{2j-1})=M(P_{2i},P_{2j})=\delta_{ij} \frac{2 i \pi}{k+1} \left(\frac{\left|D\right|}{\pi}\right)^{i}.
\end{equation}
The formula for ellipses is slightly more involved, but also explicit (see \cite[Prop. 4.7]{AMMARI-KANG-07}).
In this case, an algebraic approach (as described above) would provide exact results.  Thus we use these shapes as a first benchmark for more 
complex geometries, for which the analytic formulae do not exist or are unknown.  We  compute the integrals appearing in Section~\ref{sec:linear} using a 
midpoint rule and fixing $N_1=N_2$. In a first experiment, for a perfect disk of radius $.5$, we study the influence of the resolution of the boundary discretization and the number of polynomials used 
in the basis to compute the contracted GPT up to order 4.
\begin{figure}[h]
\includegraphics[width=0.5\linewidth]{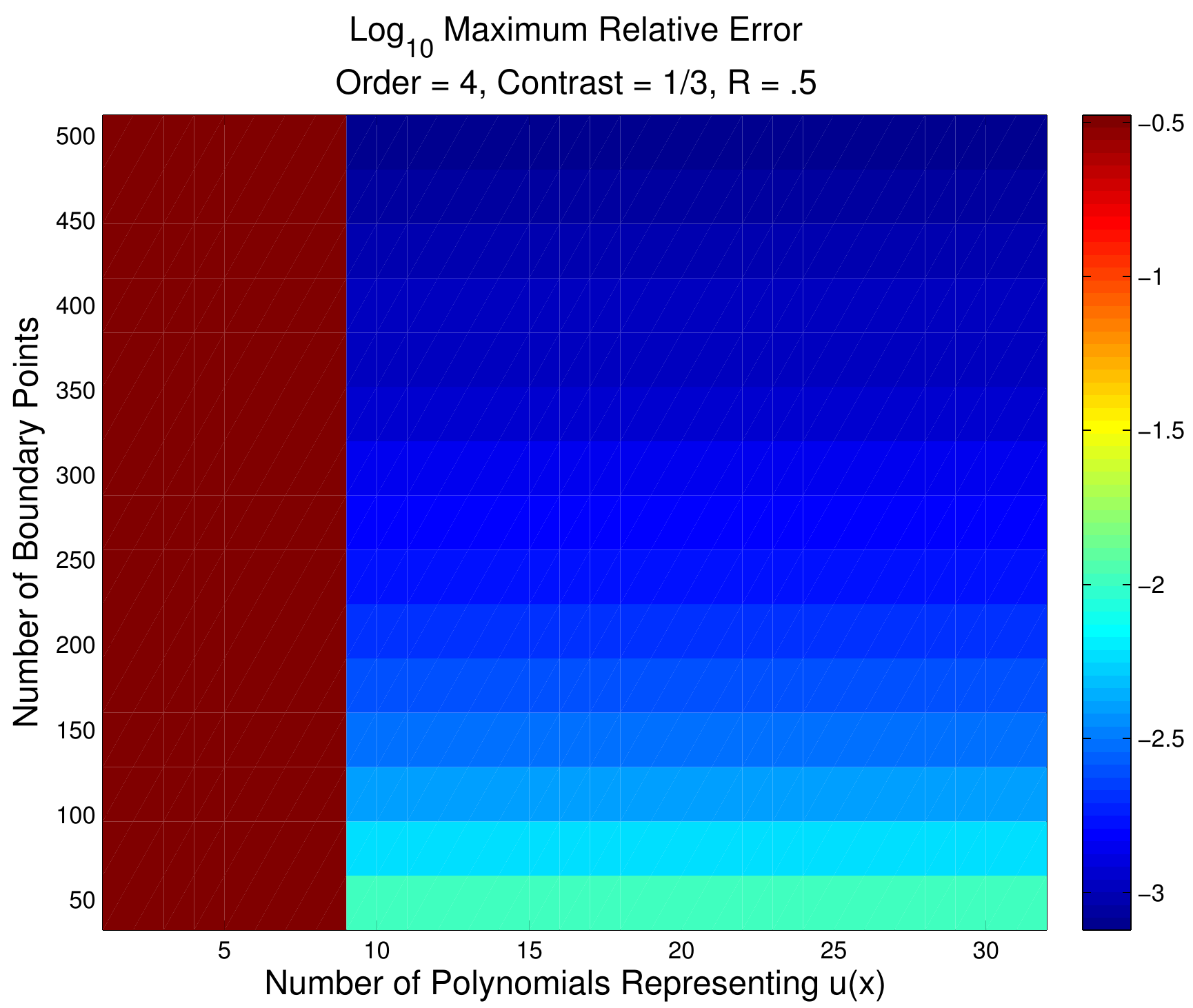}
\caption{\label{fig:disk} Relative error for GPT up to order $4$.}
\end{figure}
Figure ~\ref{fig:disk} shows a plot of the result. The value represented is the maximal relative absolute difference between the approximate tensor and the exact tensor.  
For a tensor $M'$ approximating $M$, the relative error  is given by 

\begin{equation} 
\epsilon =\displaystyle\max\limits_{i,j} \frac{ \left| M_{ij}'-M_{ij} \right| }{\displaystyle\max\limits_{k,l \geq \min(i,j)} \left| M_{kl} \right|}.
\end{equation} 
As $i$ and/or $j$ increases,  $\max\limits_{k,l \geq \min(i,j)} \left| M_{kl} \right|$ decreases: this scaling factor compensates the variation 
of tensor's coefficient size.
Figure~\ref{fig:disk} shows that the error behaves as we could expect. If the number $N$  of harmonic polymonial used is less than twice the order 
(in these results, $n=4$), the results are incorrect.  Note that the harmonic polynomials come in pairs as real and imaginary parts of $(x_1+ix_2)^m$, thus it makes sense to expect accurate 
results for tensors of order $n$ only if the maximal degree ($m$ in this case) of the representation of $u$ is greater than $n$ ($m>n$).
Once this threshold is passed, the error decreases with the number of discretization points independently of the number of polynomials used.

With the same geometry and the same order, we use $9$ harmonic polynomials, and test the behavior of this method as the contrast increases, for various discretizations. 
The result is shown in  Figure~\ref{fig:contrast}. Higher contrast $k>1$ leads to greater errors, but the error remains bounded when the contrast tends to zero.
\begin{figure}[h]
\includegraphics[width=0.5\linewidth]{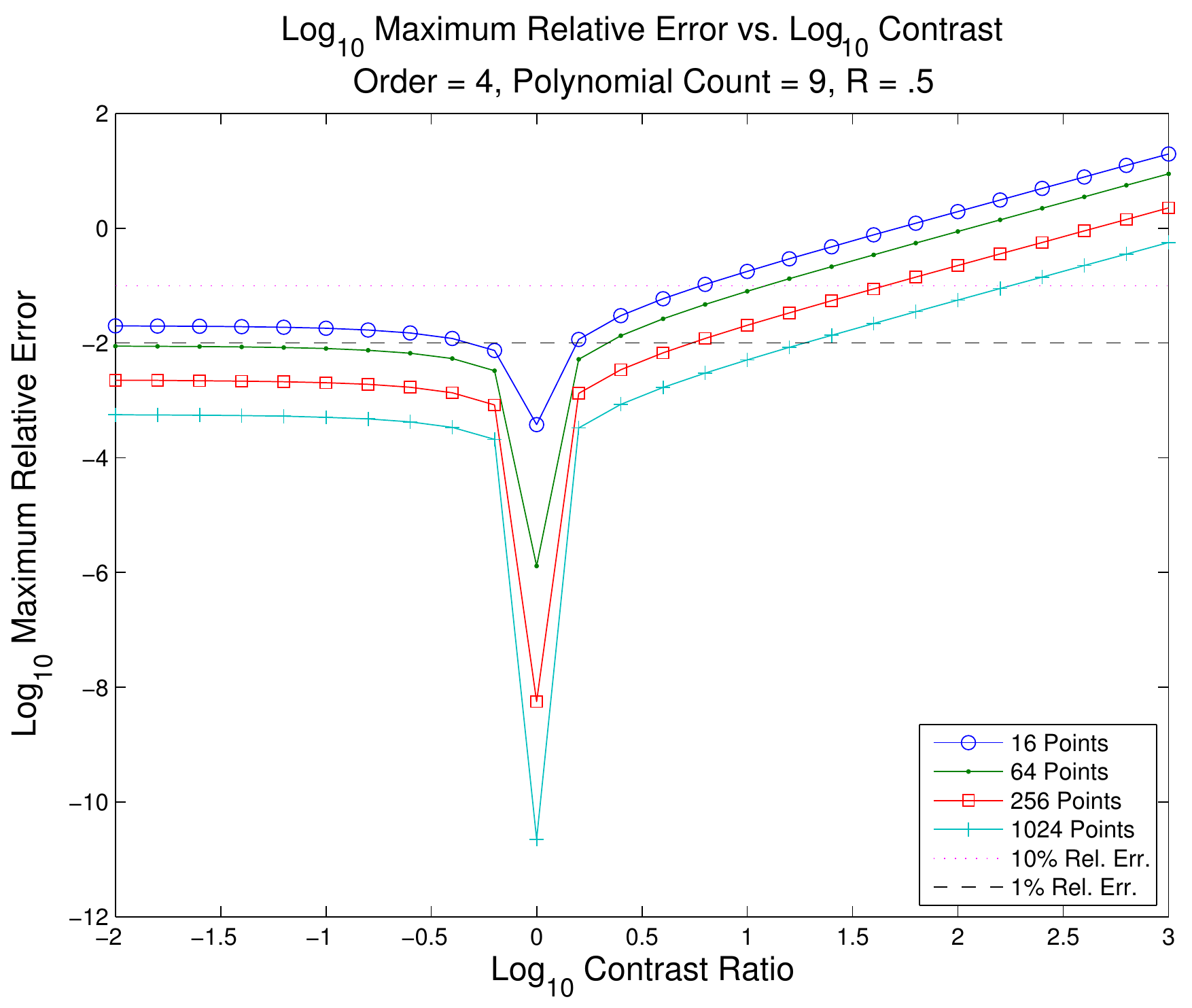}
\caption{Error vs. conductivity ratio.  The error is bounded for $k \leq 1$ and increases with $k>1$. 
For a point-count of $1024$ and $k\approx 10$, the  relative error is less than $1\%$.}
\label{fig:contrast}
\end{figure}
With the same geometry, and a contrast of $1/3$, we now fix the number of polynomials used as a function of the order 
($2n+1$ polynomials, where $n$ is the tensor order) and increase the order.
The result is shown in Figure~\ref{fig:order}. The image on the right represents the same data as that on the left but with a different scale on the vertical axis 
to show more detail. 
When too few points are used to discretize the boundary, that is,  $16$ or $32$ points, the left plot shows that the ill-conditioning of the linear  system causes extremely
divergent results. For the other cases, the results are consistent.
For tensors up to order $28$, a parameterization point-count of $256$ points gives relative errors of less than $1$ percent for this disk case.
\begin{figure}[h]
\includegraphics[width=0.47\linewidth]{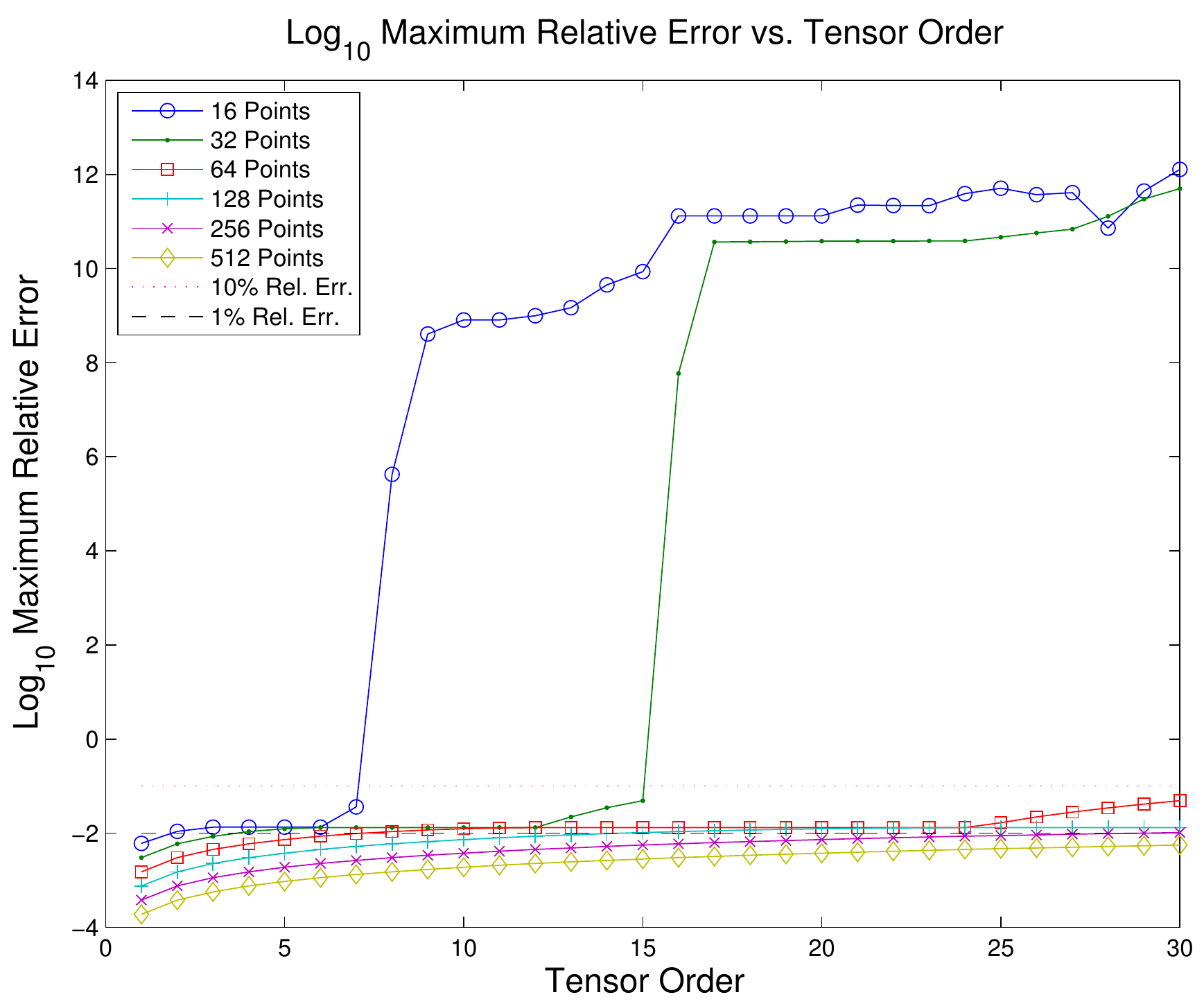}
\includegraphics[width=0.47\linewidth]{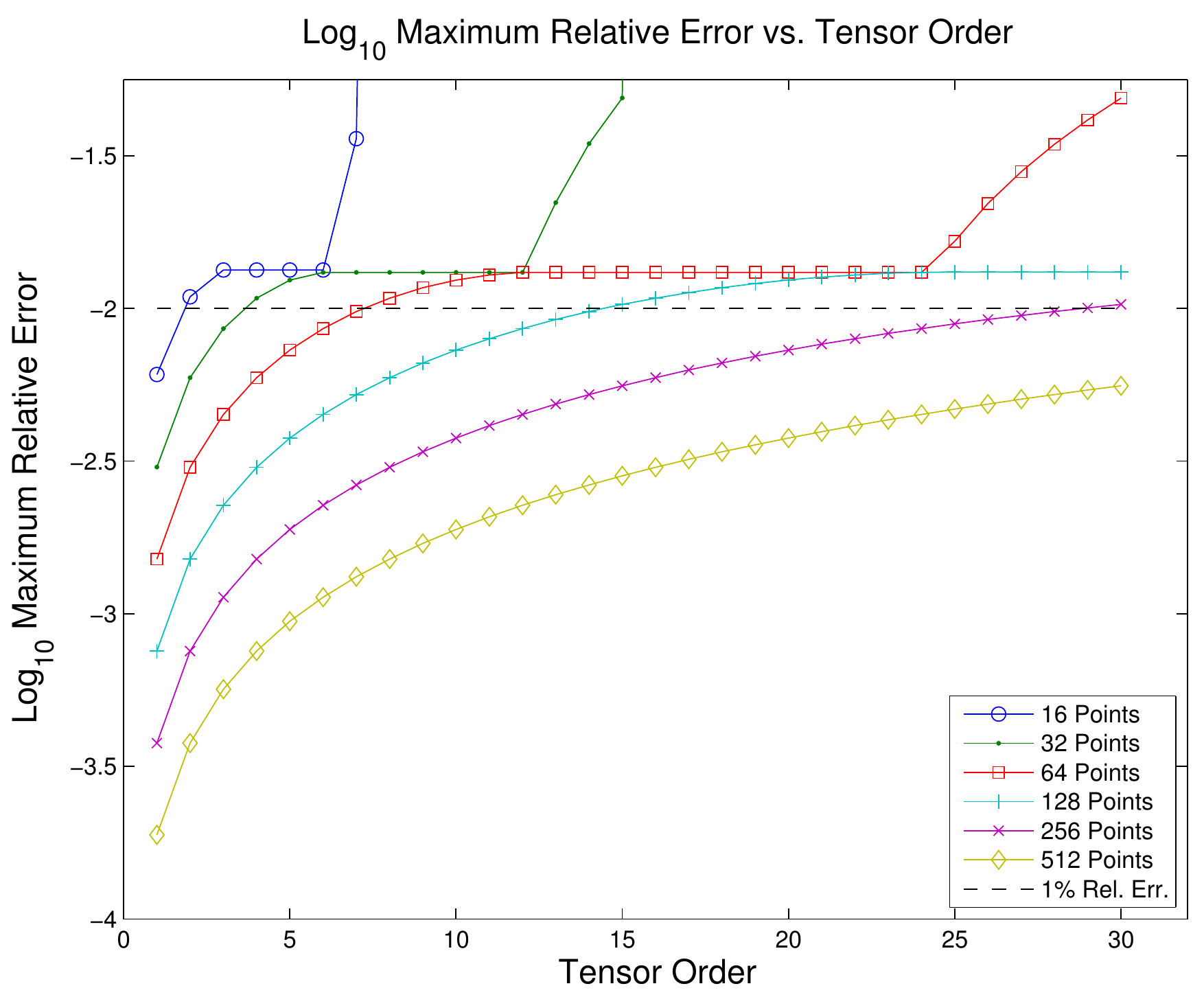} 
\caption{Results for higher order tensors for a disk of radius one and conductivity ratio equal to $1/3$.}
\label{fig:order}
\end{figure}

\bigskip{}

Moving away from a perfect disk, we now investigate how the algorithm fares with ellipses of high eccentricity.  
Again, fourth order tensors were calculated and the same error measures were used.  The ellipse with semiaxis 
lengths of $a=.005$ and $b=.5$ gives vastly different results compared to the results from the disk.  
The first difference to note is that with 32 points, the approximation is not even close to the exact tensor.  

The first plot in Figure~\ref{fig:ellipse} shows the expected general trend that for larger polynomial-count 
and boundary point-count, error is smaller.  However, there is an interesting phenomenon for the case of high 
point-count; as the number of polynomials representing $u$ and $v$ increases, error attains a minimal value but 
then after a certain point begins to increase.  This is most likely due to floating point errors in Matlab that 
round the high-degree polynomials off and then add this error across each boundary segment at points on the more 
pointed edges of the ellipse.  

The second plot in Figure~\ref{fig:ellipse} gives another clear example of the 
inaccuracies of the ellipse for 32 boundary points, the increased error for $k>0$, and the existence of an "optimal" 
polynomial-count for parameterizations of high point-count.  Also, the difference between $k<1$ and $k>1$ is clear in this case, as can be seen in the second plot.

\begin{figure}[h]
\includegraphics[width=0.5\linewidth]{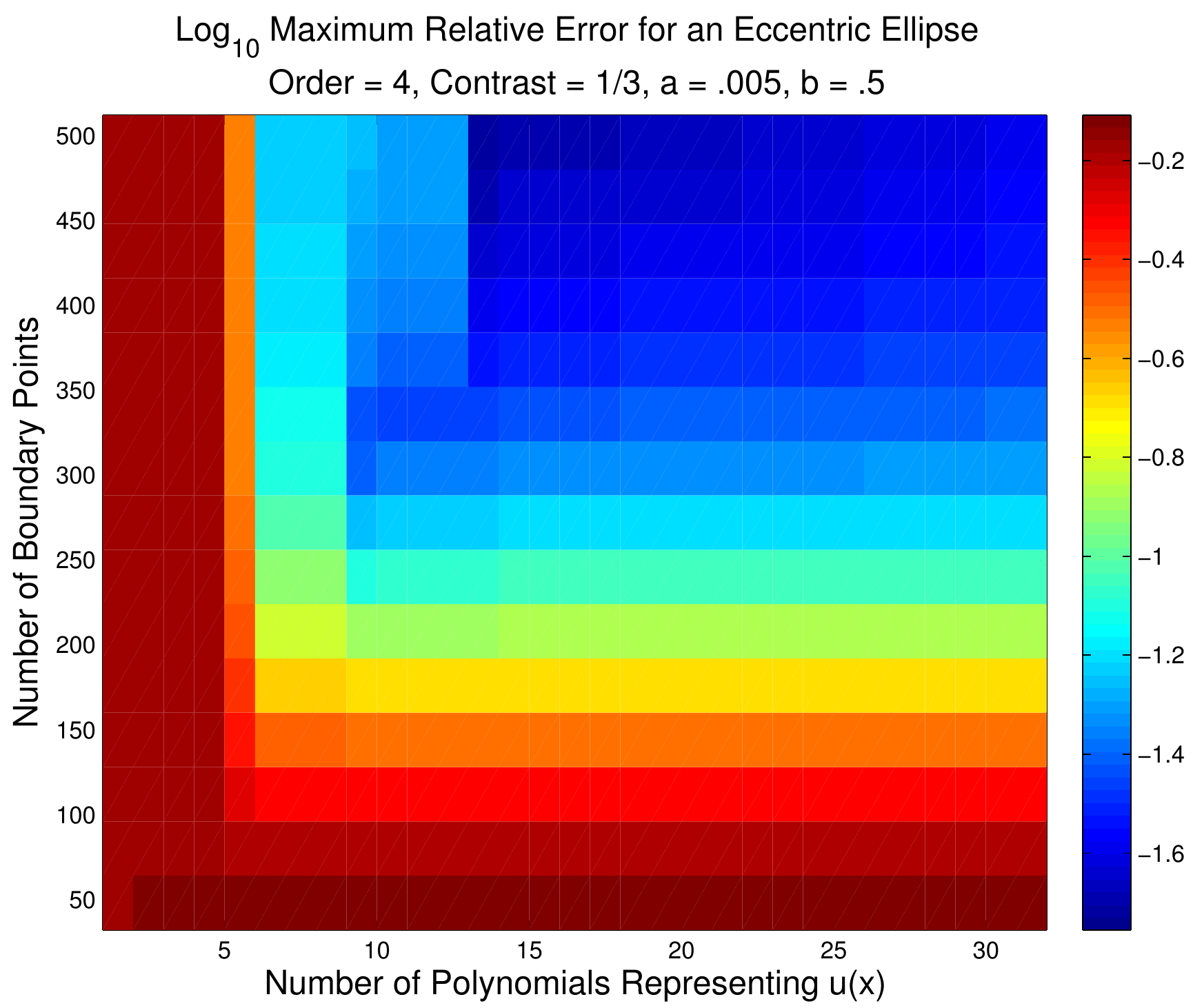}
\includegraphics[width=0.5\linewidth]{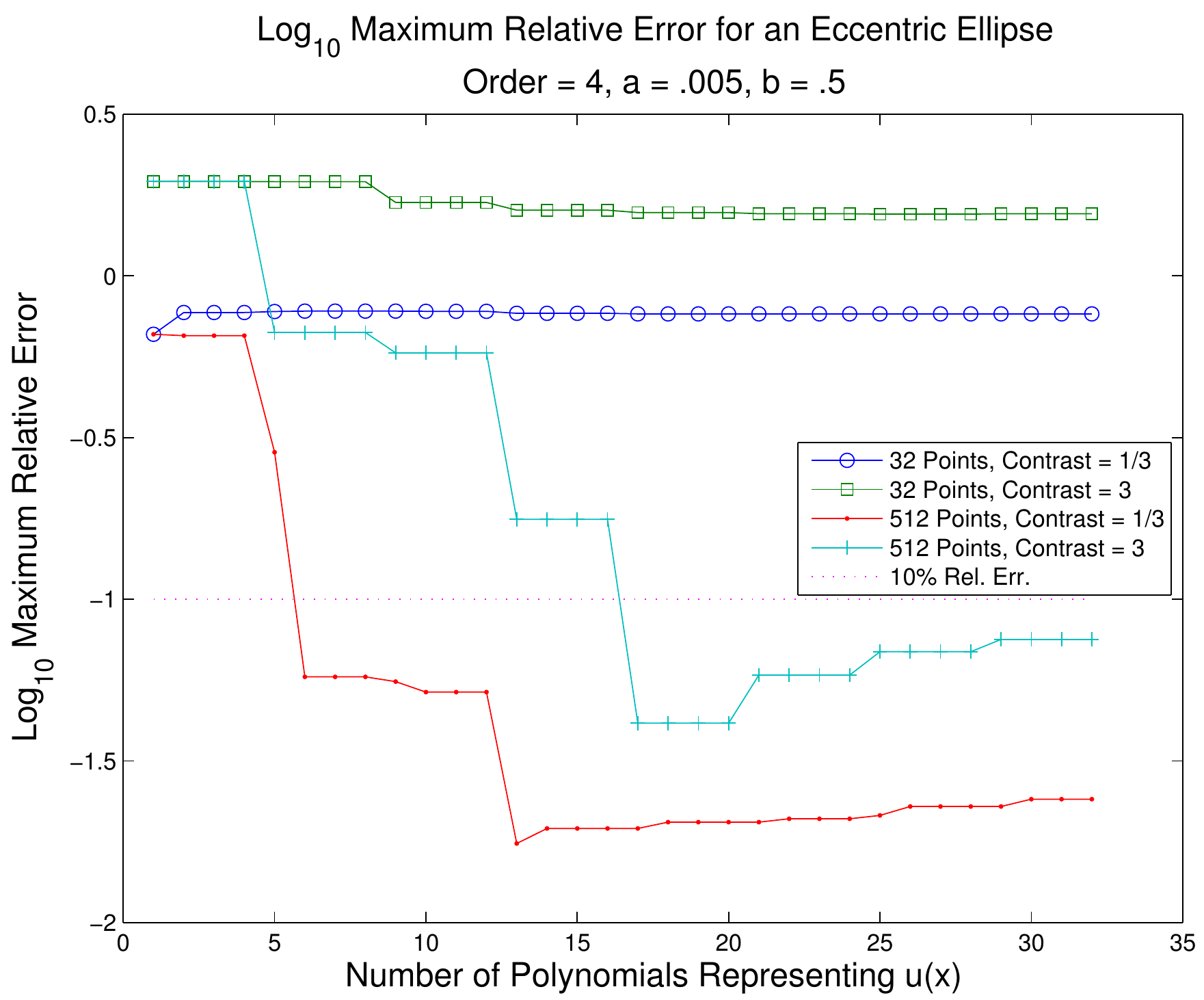}
\caption{Results for ellipse of considerable eccentricity with semiaxes $a = .01$ and $b = 1$.  The general region where the relative error is below $10$ percent is where 
polynomial-count is greater than $10$ and point-count is greater than $200$. }
\label{fig:ellipse}
\end{figure}

Lastly, we considered the amount of time taken to calculate tensors for the unit disk.  
On a 64-bit, 2.3 GHz machine, Matlab computed fourth order tensors according to the results given in Table ~\ref{table:time}.

\begin{table}[ht]
\caption{Calculation Time (seconds): Timing results for a tensor of order $1$ and $10$ each using two different polynomial-counts: one equal to the minimal required amount, and one equal to twice the minimal required amount.}
\centering 
\begin{tabular}{c c | c c c} 
\hline\hline 
Tensor Order & Number of Polynomials & 16 Points & 256 Points & 1024 Points \\ [0.5ex] 
\hline 
1 & 3 & .0149 &.0340 & .8890  \\ 
 & 6 & .0188 & .0467 & 1.1463 \\
\hline 
10 & 21 & .0308 &.1063 & 2.7794 \\
 & 42 & .3055 & .4712 & 5.4148 \\
\hline
\end{tabular}
\label{table:time} 
\end{table}

\section{\label{sec:gui}A Matlab package with graphical user interface}
To extend the algorithm, we have developed a Matlab graphical user interface that calculates contracted GPT's of any shape. 
The interface provides an efficient and user-friendly implementation of the algorithm described above and allows the user 
to easily change parameters used in the calculation.  In this context, it also serves as a useful tool for further investigation 
into the accuracy of the computation for disks and ellipses. It is available from the Matlab Central File Exchange server under 
the name 'Myriapole'.

\begin{figure}[h]
\includegraphics[scale=.3]{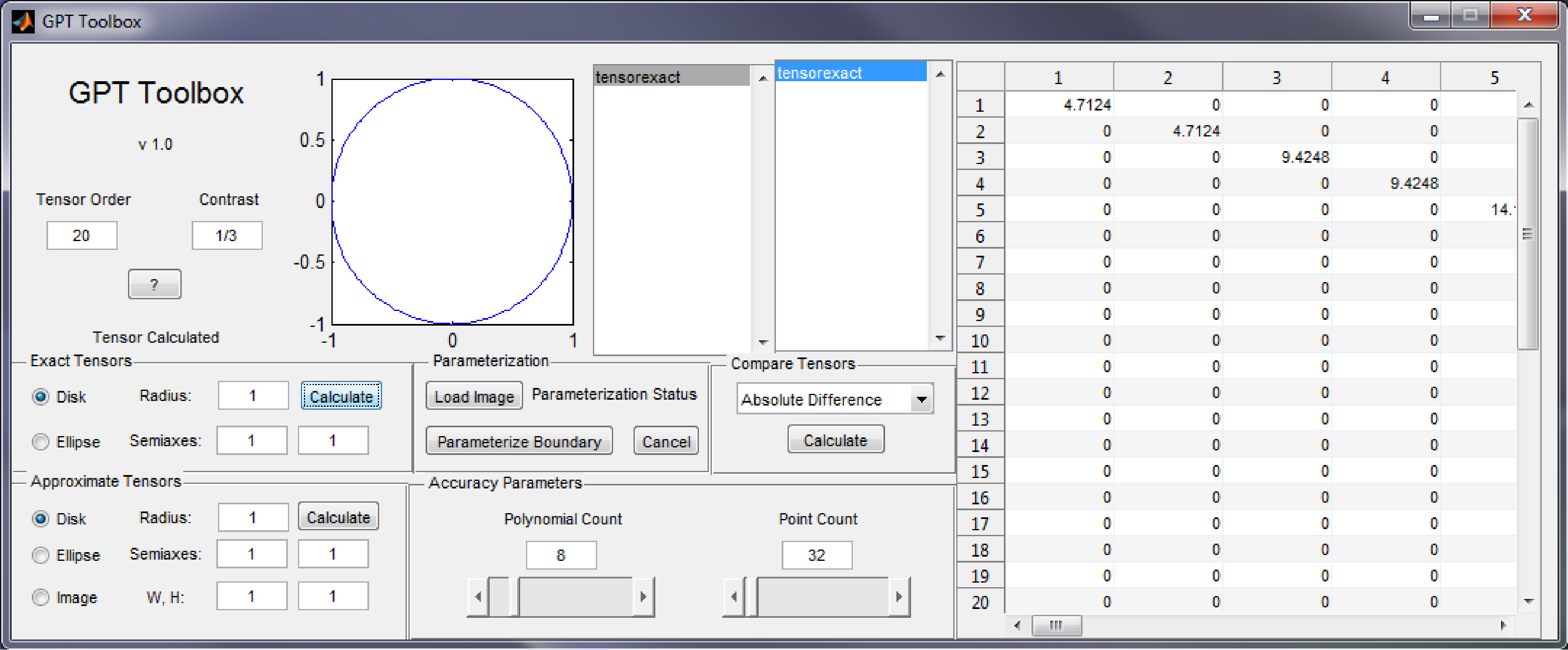}
\caption{The layout of the graphical user interface Myriapole.}
\label{fig:gui}
\end{figure}

The principal parameters of contracted GPT's are its order and conductivity contrast.  
Order in this situation refers to the highest power of $x_1+ix_2$ appearing in the set of polynomials $P$ used 
for the calculation of $M(P_i,P_j)$.  Therefore, a contracted tensor of order $n$ will be a matrix of size 
$2n \times 2n$, where each odd row and column correspond to $P_m = \Re ((x_1+ix_2)^m)$ while each even row and 
column correspond to $P_m = \Im ((x_1+ix_2)^m)$.   The tensor order and contrast, along with the boundary 
shape and size, are referred to as tensor properties and are separate from the other "approximation" parameters in the graphical user interface (GUI), 
namely the number of polynomials used to represent $u$ or the number of points in the parameterization of the 
boundary; the tensor properties change the values of the exact tensor while the approximation parameters change the accuracy of the 
approximating tensor.

The  panels of the GUI separate its different functions. For disks and ellipses, one can use the interface 
to calculate both exact and approximate tensors.  Exact tensors are calculated using analytical formulas 
while the approximate ones use the algorithm described in this paper. This is useful for testing the effects 
the input parameters have on the accuracy, and in fact a separate panel in the interface can be used to
 easily assess the errors of an approximately calculated GPT.   For any exact tensor $M$ and approximating
 tensor $M'$, the included measures viewable in the GUI are the absolute difference $|M'-M|$, relative absolute difference
 and $L^p$ matrix norms $\left(\sum_{ij}  |M_{ij}'-M_{ij}|^p\right)^\frac{1}{p}$ for $p=1,2,$ and $\infty$.

Beyond the capability to test the algorithm using disks and ellipses, the GUI also allows the user to import 
and calculate the GPT of  simple shapes described in a bitmap image file.  The easiest way to calculate tensors 
for shapes in this manner is to draw the filled shape in black on some white background in a graphic utility 
and then import this image into the interface.  A quick algorithm parameterizes the shape and thus allows the 
user to calculate its approximate tensor.   

\section{Concluding remarks}
We have introduced a simple method to compute contracted Generalized Polarization tensors, together with a 
graphical user  interface to test it. We verified on benchmark cases, disks and ellipses, that the resolution 
method was accurate, and was quite robust in the case of ellipses with very large aspect ratio.
This code is available to anyone, and can be downloaded from the Matlab Central file exchange repository. 
Paired with other routines that complete the asymptotic expansions  of domains marked with inhomogeneities 
of contrasting conductivity, this Graphical User Interface or the underlying code can function as a useful 
tool in implementations topological optimization methods, inverse conductivity tomography algorithms, 
calculations involving effective conductivity of dilute composites, and others. It is possible that this 
approach can be extended to compute anisotropic generalized polarization tensors. We hope that this 
semi-algebraic tensor calculation method will inspire similar numerical methods for related transmission 
and layer potential problems.

\section*{Acknowledgements }
Yves Capdeboscq is supported by the EPSRC Science and Innovation award to the Oxford Centre for Nonlinear PDE (EP/E035027/1). 
This work was completed in part while Anton Bongio Karrman was visiting OxPDE during a Summer Undergraduate Research Fellowship (SURF) 
awarded by the California Institute of Technology in 2010, and he would like to thank the Centre for the wonderful time he had there.

 \bibliographystyle{plain}
 \bibliography{/home/capdeboscq/Desktop/Dropbox/TeX/Mybib}

\end{document}